\documentclass[11pt]{amsart}

\usepackage{amssymb,amsmath,amsthm,latexsym,color, graphicx}
\usepackage{pifont}
\usepackage{verbatim}
\usepackage{amsfonts}

\title[The Cauchy problem  for the 1-D    quadratic fractional heat equation]
{Remarks on the Cauchy problem  for the one-dimensional    quadratic (fractional) heat equation}
\author[L. Molinet and S. Tayachi] {Luc Molinet and Slim Tayachi}
\address{\newline Luc Molinet\newline Laboratoire de Math\'ematiques et Physique
Th\'eorique\newline Universt\'e Francois Rabelais Tours,\newline
CNRS UMR 7350- F\'ed\'eration Denis Poisson \newline Parc Grandmont, 37200
Tours, France\newline e-mail: {\tt Luc.Molinet@lmpt.univ-tours.fr}
\vspace{1cm}\newline Slim Tayachi\newline Department of Mathematics
\newline Faculty of Science of Tunis\newline University Tunis El
Manar\newline 2092 Tunis,  Tunisia
\newline e-mail: {\tt slim.tayachi@fst.rnu.tn}}

\font\TenEns=msbm10 \font\SevenEns=msbm7 \font\FiveEns=msbm5
\newfam\Ensfam
\def\Ens{\fam\Ensfam\TenEns}
\textfont\Ensfam=\TenEns \scriptfont\Ensfam=\SevenEns
\scriptscriptfont\Ensfam=\FiveEns
\def\R{{\Ens R}}

\def\N{{\Ens N}}
\def\Z{{\Ens Z}}
\def\T{{\Ens T}}

\newtheorem{To}{Theorem}
\newtheorem{prop}{Proposition}[section]

\newtheorem{lem}[prop]{Lemma}
\newtheorem{rem}[prop]{Remark}
\newtheorem{defini}[prop]{Definition}

\begin{document}

\maketitle

\medskip

\begin{abstract}
We prove that the Cauchy problem associated with  the one dimensional quadratic
 (fractional) heat equation: $u_t=D_x^{2\alpha} u \mp u^2,\; t\in (0,T),\; x\in \R$ or $ \T $,
with $ 0<\alpha\le 1 $ is well-posed in $ H^s $ for $ s\ge \max(-\alpha,1/2-2\alpha) $ except in the
case $ \alpha=1/2 $ where it  is shown to be  well-posed for $ s>-1/2 $ and ill-posed for $ s=-1/2 $.
As a by-product we
improve  the known well-posedness
 results for  the heat equation ($\alpha=1$) by reaching the  end-point Sobolev  index $ s=-1 $.
 Finally, in the case $ 1/2<\alpha\le 1 $, we also prove
 optimal results in the Besov spaces $B^{s,q}_2.$
\end{abstract}

\medskip

 {\it Keywords:} Nonlinear heat equation, Fractional heat equation, Ill-posedness,
 Well-posedness, Sobolev spaces, Besov spaces.

\vspace{0.2cm} {\it 2000 AMS Classification:} 35K15, 35K55, 35K65,
35B40

\vspace{0.5cm}

\section{Introduction and main results}
The Cauchy problem  for the quadratic fractional heat equation reads \setcounter{equation}{0}
\begin{equation}
\label{NLH} u_t -  D_x^{2\alpha} u  =\mp  u^2,
\end{equation}
\begin{equation}
\label{NLHinitial} u(0,\cdot) =u_0,
\end{equation}
where $u=u(t,x)\in \R\; , \alpha \in ]0,1],\; t\in (0,T),\; T>0,\;
x\in\R$ or $ \T $ and $ D^{2\alpha}_x$ is the Fourier multiplier by
 $ |\xi|^{2\alpha} $.
In this paper, we
consider actually   the corresponding integral equation which is given by
\begin{equation}
\label{NLHint} u(t) = S_{\alpha}(t)u_0 \mp \int_0^t
S_{\alpha}(t-\sigma) \big(u^2(\sigma)\big) d\sigma,
\end{equation}
where $S_{\alpha}(t)$ is the linear fractional  heat semi-group and
are interested in  local well-posedness and ill-posedness  results
in the Besov spaces $ B^{s,q}_2(K) $ with $ s\in \R $, $ q\in
[1,\infty[ $ and  $ K=\R$ or $ \T.$

Let us recall that the Cauchy problem associated with the nonlinear heat equation in $ \R^n $
\begin{equation}\label{HE}
u_t - \Delta u =\mp u^k
\end{equation}
 has been studied in many papers (see for instance \cite{BC,G,Gi,HW,H,MRY,R1,R2,STW1,STW2,W1,W2,W3,W}
 and references therein). It is well-known that this equation is invariant by  the space-time dilation symmetry
  $ u(t,x) \mapsto u_\lambda(t,x)= \lambda^{2\over k-1} u(\lambda^2 t, \lambda x) $ and that
  the homogeneous Sobolev space $ \dot{H}^{\frac{n}{2}-\frac{2}{k-1}} $ is invariant by the
  associated space dilation symmetry $ \varphi(x) \mapsto \lambda^{2\over k-1} \varphi(\lambda x) $.
  The Cauchy problem \eqref{HE} is known to be well-posed in $H^s $ for $ s\ge s_c= \frac{n}{2}-\frac{2}{k-1} $
  except in the case $ (n,k)=(1,2)$. Indeed, in this case the well-posedness is only known in $ H^s $
  for $ s>-1 $ and in \cite{MRY} it is proven that the flow-map cannot be of class $ C^2 $ below $ H^{-1}$.
  Hence,  this result is close to be optimal if one requires the smoothness of the flow-map. Recently,
  it was proven in \cite{IO} that the associated solution-map : $ u_0 \mapsto
   u $ cannot be  even continuous in $ H^s $ for $ s<-1 $. The first aim of this work is to push down
   the well-posedness result to the end point $ H^{-1}$.  The second step is to extend these type of results
    for the one-dimensional quadratic fractional heat equation \eqref{NLH}. Indeed we will derive optimal
     results for the Cauchy problem \eqref{NLH} in the scale of the Besov spaces
   $ B^{s,q}_2 $ in the case $ \frac{1}{2}< \alpha \le 1 $. In particular we will prove that the lowest
   reachable Sobolev index is $ -\alpha $ that is strictly bigger then the critical Sobolev index for
    dilation symmetry that is $ 1/2-2\alpha $. \\
To reach the end-point index $ H^{-\alpha}$ we do not follow the
classical method for parabolic equations (cf. \cite{G,R2,W}) that
does not seem to be applicable here. We rather rely on an approach
that was first introduced by Tataru \cite{tataru} in the context of
wave maps. Note that  we mainly follow \cite{MV} where this method
has been adapted for dispersive-dissipative equations.  The fact
that our equation is purely parabolic enables us  to simplify the
proof. The optimality of our results follows from an approach first
introduced  by Bejenaru-Tao \cite{BT} for a one-dimensional
quadratic Schr\"odinger equation. This approach is based on a
high-to low frequency cascade argument.

Finally we consider the case $ 0<\alpha\le 1/2 $. By classical parabolic methods we obtain the well-posedness
in the Sobolev space \footnote{Recall that $1/2-2\alpha $ is the critical Sobolev index for dilation symmetry.}
$ H^{s} (\R) $, $s\ge 1/2-2\alpha $, unless $ \alpha=1/2$.  On the other hand, following a very nice result by
Iwabuchi-Ogawa \cite{IO}, we prove that \eqref{NLH} is ill-posed in $H^{-1/2}(\R) $ for $\alpha=1/2$.
It is worth noticing that $ (1/2,-1/2) $ is the intersection of the straight  borderlines for well-posedness
that are  $s=-\alpha $ and $s=1/2-2\alpha $.

  Before stating our main result, let us give the precise definition of well-posedness we will use in this paper.
 \begin{defini}
  We will say that  the  Cauchy problem \eqref{NLH}-\eqref{NLHinitial} is  (locally)  well-posed   in some  normed
  function space $ B $  if, for any initial data $ u_0\in B $, there exist a radius $ R>0 $, a  time
  $ T>0 $  and a unique solution $ u$ to \eqref{NLHint},  belonging to some space-time function space
  continuously embedded in $ C([0,T];B) $, such that for any $ t\in [0,T] $ the map
   $ u_0\mapsto u(t) $ is continuous from the ball of $ B $ centered at $ u_0 $ with radius $ R $
    into $ B $.    A Cauchy problem will be  said to be ill-posed if it is not well-posed.
 \end{defini}
\begin{To} Let $ K=\R $ or $\T $
 and  $ \alpha \in ]1/2,1] $. The Cauchy problem \eqref{NLH} is locally well-posed in the Besov
 space $ B^{s,q}_2(K) $ if and only if  $(s,q)\in \R\times [1,+\infty[  $ satisfies
 $ s>-\alpha $ or $ s=-\alpha $ and $ q\in [1,2]$.
\end{To}
\begin{rem}
Our negative results can be stated more precisely in the following way : For any couple
 $(s,q)\in \R\times [1,+\infty[  $ satisfying,
 $ s<-\alpha$ or  $ s=-\alpha $ and $ q>2$,  there exists $ T> 0 $ such that the
 flow-map $ u_0\mapsto u(t) $ is not continuous at the origin  from  $ B^{s,q}_2(K) $
 into $ {\mathcal  D}'(K) $ for any $ t\in ]0,T[ $.
\end{rem}
This paper is organized as follows. In the next section we define
our resolution spaces in the case $ K=\R $. In Section 3 we derive
the needed linear estimates on the free term and the retarded
Duhamel operator and in Section 4 we prove our well-posedness
result. Section 5 is devoted to  the non-continuity results for the
same range of $ \alpha$. In Section 6 we complete the well-posedness
results by considering the case   $ 0<\alpha \le 1/2 $. First, by
classical parabolic methods,   we prove that we can reach  the
critical Sobolev index for dilation symmetry that is $ 1/2-2\alpha $
unless $\alpha=1/2$. Then, following  \cite{IO}, we prove that
\eqref{NLH} is ill-posed in $H^{-1/2}(\R) $ for $\alpha=1/2$.
Finally we explain the needed adaptations in the periodic case $
K=\T $.

Throughout the paper, we will write $f\lesssim g,$ whenever a
constant $C\geq 1,$ only depending on parameters and not on $t$ or
$x$, exists such that $f\leq Cg$. We write $f\sim g$ if $f\lesssim
g,$ and $g\lesssim f.$ If $C$ depends on parameters $a,$ we write
$f\lesssim_a g,$ instead.

\section{Resolution Space}\label{section2}
We use the following definition for the Fourier transform
$${\mathcal F}({f})(\xi)=\int_\R f(x)e^{-ix\xi}dx,$$
and the inverse Fourier transform is
$${\mathcal F}^{-1}({f})(x)={1\over 2\pi}\int_\R{\mathcal F}({f})(\xi) e^{ix\xi}d\xi,$$
\setcounter{equation}{0} for $f$ in ${\mathcal S}(\R),$ the Schwartz
space of rapidly decreasing smooth functions, and by duality if $f$
in ${\mathcal S}'(\R),$ the space of tempered distributions. We
denote sometimes ${\mathcal F}({f})$ by $\hat{f}.$ The fractional
power of the Laplacien can be defined by the Fourier transform: For
$\alpha\in \R,$
$${\mathcal F}\big((-\partial^2_x)^{\alpha}f\big)(\xi)=|\xi|^{2\alpha}{\mathcal F}({f})(\xi).$$ Let $s$
be a real number. The Sobolev space $H^s(\R)$ is defined by
$$H^s(\R)=\{u\in \mathcal{S}'(\R)\; |\;
\int_{\R}(1+|\xi|^2)^s|\mathcal{F}(u)(\xi)|^2d\xi<\infty\}$$ where
$\mathcal{F}(u)$ is the Fourier transform of $u.$ The norm on
$H^s(\R)$ is defined by
$$||u||_{H^s(\R)}=\Big(\int_{\R}(1+|\xi|^2)^s|\mathcal{F}(u)(\xi)|^2d\xi\Big)^{1/2}.$$
We will need a Littlewood-Paley analysis. Let $\eta\in C^\infty_0(\R)$ be a non negative even function such that  $\mbox{supp}\,  \eta\subset [-2,2]$ and $\eta\equiv 1$ on $[-1,1]$. We define  $\varphi(\xi)=\eta(\xi/2)-\eta(\xi)$ and the Fourier multipliers
$$\mathcal{F}(\Delta_ju)(\xi)=\varphi(2^{-j}\xi)\mathcal{F}u(\xi),\;
j\geq 0, \quad \mbox{ and }\quad \mathcal{F}(\Delta_{-1}
u)(\xi)=\eta(\xi)\mathcal{F}u(\xi)\; .$$ For any $ s\in \R $ and $
q\ge 1 $, the Besov space $ B^{s,q}_2(\R) $ is defined as the
completion of $ {\mathcal S}(\R) $ for the norm
$$
\|u\|_{B^{s,q}_2(\R)} =\Big(\sum_{j\geq
-1}2^{jsq}||\Delta_ju||^q_{L^2(\R)}\Big)^{1/q} \; .
$$
For  $s\in \R $, $s_1<s_2$, $1\le q_1\leq q_2$  and $ q\ge 1 $ we have the following embeddings
$$B^{s,q_1}_2\hookrightarrow B^{s,q_2}_2
\;  \mbox{ and  } \; B^{s_2,q}_2\hookrightarrow
B^{s_1,1}_2.$$
Moreover, It is well-known that the
$H^{s}(\R)$-norm  is equivalent to the $ B^{s,2}_2 $-norm
so that $ H^s(\R) = B^{s,2}_2(\R) $. \\
Finally,
for $ 1\le p \le \infty $ we consider the space-time space $\tilde{L}^p(\R; B^{s,q}_2)$ equipped with the norm
$$
 \| u\|_{\tilde{L^p_t}  B^{s,q}_2}=\Bigl[ \sum_{j\ge -1 }  2^{s  j q } \|\Delta_j u(t)
 \|_{L^p_t L^2_x}^q \Bigr]^{1/q}.
$$
We are now able to define our resolution space. For $T>0$ fixed, we  consider the space $X^{s,q}_{\alpha,T}=\tilde{L}_T^\infty
B^{s,q}_2\cap \tilde{L}_T^2 B^{s+\alpha,q}_2 $
  equipped with the norm:
$$\|u\|_{X^{s,q}_{\alpha,T}}=\Bigl[ \sum_{j} \sup_{t\in ]0,T[} 2^{s  j q } \|\Delta_j u(t) \|_{L^2_x}^q \Bigr]^{1/q}  +
\Bigl[ \sum_{j} 2^{j(s+\alpha)q}  \|\Delta_j u \|_{L^2_T L^2_x}^q \Bigr]^{1/q}
.$$ Let us also consider the space
$$Y^{s,q}_{\alpha}:= \Big\{u\in \tilde{L}^1\big(\R_+^* ; B^{s+2\alpha,q}_2(\R)\big) \mbox{ and } \partial_t u\in \tilde{L}^1\big(\R_+^*;
B^{s,q}_2(\R)\big)\Big\}
$$
 equipped with the norm:
$$\|u\|_{Y^{s,q}_{\alpha}}=\Bigl[ \sum_{j}  2^{(s+2\alpha)  j q } \|\Delta_j u\|_{L^1_t L^2_x}^q \Bigr]^{1/q}  +
\Bigl[ \sum_{j}  2^{s  j q }  \|\Delta_j u_t \|_{L^1_t L^2_x}^q \Bigr]^{1/q}
$$
For $ T>0 $,  the restriction space $ Y^{s,q}_{\alpha,T} $ of $
Y^{s,q}_\alpha $ is endowed with the usual norm
$$
\|u\|_{Y^{s,q}_{\alpha,T}} = \inf_{v\in Y}  \{ \|v\|_{Y^{s,q}_{\alpha}} \, , \, v\equiv u \mbox{ on } ]0,T[ \, \} \; .
$$
For $T>0 $  our resolution space will be
$E_{\alpha,T}^{s,q}=X^{s,q}_{\alpha,T}+Y^{s,q}_{\alpha,T}$ endowed with the usual norm for a sum space :
$$\
\|u\|_{E_{\alpha,T}^{s,q}}:=\inf_{u=v+w}(\|v\|_{X^{s,q}_{\alpha,T}}+\|w\|_{Y^{s,q}_{\alpha,T}}).$$
\section{Linear estimates}\label{linear}
We first establish  the following lemma.
\begin{lem}
\label{p1} Let $0<T\le 1  $ and $\varphi\in B^{s,q}_2.$ Then we have
\begin{equation}
\label{2'} \|S_{\alpha}(t)\varphi\|_{X^{s,q}_{\alpha,T}}\lesssim
\|\varphi\|_{ B^{s,q}_2} \; .
\end{equation}
\end{lem}
\begin{proof}
The standard smoothing effect of the (fractional) heat semi-group is not sufficient here
since we have
$$\|S_{\alpha}(t)\varphi\|_{B^{s+\alpha,q}_2}\lesssim
t^{-\frac{1}{2}}\|\varphi\|_{B^{s,q}_2}$$ and the right hand side of
this inequality is not square integrable near $t=0.$ Integrating by
parts the linear fractional heat equation
\begin{equation}\label{linearheat}
\partial_tu-D^{2\alpha}_x u=0
\end{equation}
on $]0,t[\times \R $, $t>0$, against $ u $ and  using that $ u(0)=\varphi $, we obtain
$$
\int_{\R} u^2(t,x) \, dx+\int_0^t \int_{\R} |D^{\alpha}_x u (s,x)|^2 \, dx\, ds = \int_{\R} \varphi^2(x) \, dx \;.
$$
Using that for each $ j\in \N $, $ \Delta_j S_\alpha (t)
D^{s}_x\varphi $ satisfies the linear fractional  heat equation
\eqref{linearheat} with
 $ D^{s}_x \varphi $ as initial datum, powering in $ q/2$ and then summing in $ j\ge 0 $, we get for any $ T>0 $,
$$
\Bigl( \sum_{j\ge 0}2^{s j q }\|\Delta_j
S_{\alpha}(t)\varphi\|_{L^\infty_T L^2_x}^q\Bigr)^{1/q} + \Bigl(
\sum_{j\ge 0}2^{(s+\alpha) j q }\|\Delta_j
S_{\alpha}(t)\varphi\|_{L^2_T L^2_x}^q \Bigr)^{1/q} \lesssim \Bigl(
\sum_{j\ge 0}2^{s j q}\| \Delta_j \varphi\|_{L^2_x}^q \Bigr)^{1/q}
\; .
$$
On the other hand, for $ j=-1 $ we write
$$
\|\Delta_{-1}  S_{\alpha}(t)\varphi\|_{L^\infty_T
L^2_x}+\|\Delta_{-1}  S_{\alpha}(t)\varphi\|_{L^2_T L^2_x}\le 2\,
T^{1/2} \|\Delta_{-1}  S_{\alpha}(t)\varphi\|_{L^\infty_T L^2_x} \le
2 \,  T^{1/2}\|\Delta_{-1}  \varphi\|_{L^2_x}
$$
and the result follows.
\end{proof}
As a direct consequence we get the following estimate on the semi-group :
 Let $ 0<T\le 1  $ and $\varphi\in B^{s,q}_2 $ then it holds
\begin{equation}\label{2}
\|S_{\alpha}(t)\varphi\|_{E_{\alpha,T}^{s,q}}\lesssim \| S_{\alpha}(t)\varphi\|_{X^{s,q}_{\alpha,T}}\lesssim
\|\varphi\|_{B^{s,q}_2}\; .
\end{equation}

Let us now define the operator $\mathcal{L}_{\alpha}$ by

\begin{equation}
\label{definL} \mathcal{L}_{\alpha}(f)(t,x)=\int_0^tS_{\alpha}(t-t')f(t')dt'.
\end{equation}

Then we have
\begin{lem} \label{p3} Let $ 0<T\le1 $ and $f\in E_{\alpha,T}^{s,q}.$
Then we have
\begin{equation}
\label{1'} \|\mathcal{L}_{\alpha}(f)\|_{Y^{s,q}_{\alpha,T}}\lesssim
(1+T) \|f\|_{\tilde{L}^1_T B^{s,q}_2}\; .
\end{equation}
\end{lem}
\begin{proof}
It suffices to prove  the result for a time extension of $ \mathcal{L}_{\alpha}(f) $. More precisely, it suffices to prove that
$$
\|\eta \mathcal{L}_{\alpha}(f)\|_{Y^{s,q}_{\alpha,T}}\lesssim
\|f\|_{\tilde{L}^1_{t>0} B^{s,q}_2}\; ,
$$
 for any $ f\in \tilde{L}^1_{t>0} B^{s,q}_2 $ supported in time in $ [0,1] $ and where $ \eta \in C_0^\infty(\R) $ is defined in Section \ref{section2} .
Let $u$ be the solution of the Cauchy problem
$$\partial_t  u- D_x^{2\alpha} u =f,\; u(0)=0.$$
It is easy to check that $ u = \mathcal{L}_\alpha(f)$.
Multiplying this equation by $u$ and integrating by parts, we get
$$\frac{1}{2} {d\over dt}\int_\R u^2+\int_\R(D_x^\alpha u)^2=\int_\R fu \; .$$
Applying this equality to localizing in frequencies equation and using Bernstein inequality and the
Cauchy-Schwarz one, we get for any $ j\in \N $,
$$\frac{1}{2} {d\over dt}\int_\R u_j^2+2^{2\alpha j}\int_\R u_j^2\leq \big(\int_\R f_j^2\big)^{1/2}\big(\int_\R u_j^2\big)^{1/2}.$$
Here $u_j=\Delta_j u, \; f_j=\Delta_ j f$. If  $\big(\int_\R
u_j^2\big)^{1/2}\neq 0 $  we divide this last inequality by
$\big(\int_\R u_j^2\big)^{1/2}$ to obtain
$${d\over dt}\Big(\big(\int_\R u_j^2\big)^{1/2}\Big)+2^{2\alpha j}\big(\int_\R u_j^2\big)^{1/2}
\leq \big(\int_\R f_j^2\big)^{1/2}.$$
 On the other hand, the smoothness and non negativity  of $t\mapsto \|u_j(t)\|_{L^2_x}^2 $ forces $ \frac{d}{dt} \|u_j(t)\|^2_{L^2_x}=0 $ as soon as $ \|u_j(t)\|_{L^2_x} = 0 $. This ensures that  the above differential inequality is actually valid for all $ t>0 $.  Integrating this
differential inequality in time we get for any $ j\in \N $,
\begin{equation}
\ 2^{2\alpha j} \|u_j\|_{L^1_tL^2_x}\lesssim \|f_j\|_{L^1_tL^2_x}.
\end{equation}
Now, in the case $ j=-1$, we get in the same way $ \|\Delta_{-1} u
\|_{L^2} \lesssim  \|\Delta_{-1} f\|_{L^1_tL^2_x} $. Integrating on
$ [0,2T]$
 this leads to $  \|\eta   \Delta_{-1} u \|_{L^1_t L^2_x}\lesssim  T \|f_j\|_{L^1_tL^2_x} $. \\
Finally, in view of  the linear fractional heat equation, the  triangle inequality  leads to
\begin{equation}
\label{e} \|\partial_t (\eta u_j)\|_{L^1_tL^2_x}\lesssim \|\eta \partial_t  u_j\|_{L^1_tL^2_x}+ \| u_j\|_{L^1_tL^2_x}\lesssim
\|f_j\|_{L^1_tL^2_x}.
\end{equation}
Since $u=\mathcal{L}_\alpha(f)$, summing in $ j\in \N $ using Bernstein inequalities  and recalling  the expression of the norm in $Y^{s,\alpha}$, we
 conclude that
\begin{equation}
\label{1'-noncomplete} \|\eta
\mathcal{L}_{\alpha}(f)\|_{ Y^{s,q}_{\alpha,T}}\lesssim
\|f\|_{\tilde{L}^1_t B^{s,q}_x}\; .
\end{equation}
\end{proof}
\begin{lem} \label{popo}
Let $ 0<T\le 1  $ and $ u\in Y^{s,q}_{\alpha,T} $. Then  it holds
\begin{equation}
\|u\|_{X^{s,q}_{\alpha,T}}=\|u\|_{\tilde{L}^\infty_T B_2^{s,q}}+ \|u\|_{\tilde{L}^2_T B^{s+\alpha,q}_2} \lesssim  \|u\|_{ Y^{s,q}_{\alpha,T}} \; .
\end{equation}
In particular, $ E^{s,q}_{\alpha,T} \hookrightarrow X^{s,q}_{\alpha,T} $.
\end{lem}
\begin{proof}
Again it suffices to prove this estimate for the non restriction spaces. Actually, by localizing in space frequencies it suffices to prove that for any function
 $ u\in L^1(\R_+;L^2(\R)) $ with $ u_t \in  L^1(\R_+; L^2(\R)) $ it holds
 \begin{equation}\label{fd}
 \|u\|_{L^\infty_t L^2_x} \lesssim \|u_t \|_{L^1_t L^2_x} \quad \mbox{ and  } \quad
 \| u\|_{L^2_{t>0} L^2_x}^2 \lesssim \| u\|_{L^1_{t>0}L^2_x}  \| u_t\|_{L^1_{t>0}L^2_x}\; .
 \end{equation}
 Indeed, applying \eqref{fd} to the space frequency localization $ u_j $ of $ u $, Bernstein's inequalities lead to
 $$
  2^{js}\|u_j\|_{L^\infty_t L^2_x} \lesssim 2^{js} \|\partial_t  u_j \|_{L^1_t L^2_x} \quad \mbox{ and  } \quad
  2^{j q(s+\alpha)}\| u_j\|_{L^2_{t>0} L^2_x}^q \lesssim 2^{j q (s+2\alpha)/2} \| u_j\|_{L^1_{t>0}L^2_x}^{q/2}
   2^{j  q s/2} \| \partial_t u_j \|_{L^1_{t>0}L^2_x}^{q/2}\; ,
  $$
  which yields the result by summing in $ j $ and applying Cauchy-Schwarz in $ j $ on the right-hand member of the second inequalities.\\
  Let us now prove \eqref{fd}. The first part is a direct consequence of the equality $ u(t) =-\int_t^\infty u_t(s) ds $
   and Minkowsky integral inequality. To prove the second part we notice that  $ u^2(t) = - u(t)\int_t^\infty u_t(s) ds $ so that  we can write
  \begin{eqnarray*}
  \int_0^\infty \int_{\R} u^2(t,x)\, dx\, dt & = & \int_0^\infty \int_{\R} u(t,x) \int_0^t u_t(s,x) \, ds \, dx \, dt \\
   & \lesssim & \int_{\R} \int_0^\infty |u(t,x)| \, dt\, \int_0^\infty  |u_t(t,x)| \, dt \, dx \\
     & \lesssim & \|\int_0^\infty |u(t,\cdot)| \, dt \|_{L^2_x}  \|\int_0^\infty |u_t(t,\cdot)| \, dt \|_{L^2_x} \\
     & \lesssim &  \| u\|_{L^1_{t>0}L^2_x}  \|
     u_t\|_{L^1_{t>0}L^2_x},
  \end{eqnarray*}
where we used Minkowsky integral inequality in the last step.
\end{proof}
\section{Well-posedness for $ 1/2<\alpha\le 1$}\label{Well} \setcounter{equation}{0}
According to Lemma  \ref{p3} we easily get for $ 0<T\le 1$, 
\begin{eqnarray}
\|\mathcal{L}_{\alpha}(u^2)\|_{E_{\alpha,T}^{s,q}}  \lesssim
\|\mathcal{L}_{\alpha}(u^2)\|_{Y^{s,q}_{\alpha,T}}
& \lesssim &\Bigl( \sum_{j} 2^{j s  q} \|\Delta_j (u^2) \|_{L^1_T L^2_x}^q \Bigr)^{1/q}\nonumber \\
& \lesssim & \Bigl( \sum_{j} 2^{j q (s+1/2)} \|\Delta_j (u^2) \|_{L^1_T L^1_x}^q \Bigr)^{1/q}
 \; . \label{toto1}
\end{eqnarray}
Now, by para-product decomposition we have
$$
 \|\Delta_j (u^2) \|_{L^1_T L^1_x}\lesssim \| u\|_{L^2_{T,x}} \|\Delta_j u\|_{L^2_{T,x}} +\sum_{|k-k'|\le 3, \, k\gtrsim j} \|\Delta_{k}  u\|_{L^2_{T,x}}
  \|\Delta_{k'}  u\|_{L^2_{T,x}} \; .
  $$
  The contribution to \eqref{toto1} of the first term of the above right-hand side member can be estimated by
  $$
  \|u\|_{L^2_{T,x}} \Bigl( \sum_{j} 2^{j q (s+1/2)} \|\Delta_j u \|_{L^2_T L^2_x}^q \Bigr)^{1/q} = \|u\|_{L^2_{T,x}}\| u \|_{\tilde{L}^2_T B^{s+1/2,q}_2}
  $$
  which is acceptable as soon as $ \alpha\ge 1/2 $ and ($ s> -\alpha $  or  $ s=-\alpha $ and $1\le q\le 2 $ ). Indeed, this  last condition ensures that
  $ \|u\|_{L^2_{T,x}} \lesssim \| u \|_{\tilde{L}^2_T B^{s+\alpha,q}_2} $. For the second  term,
 we notice that  for $ \alpha>1/2 $, we can estimate its contribution   by
  $$
  \| u \|_{\tilde{L}^2_T B^{0,\infty}_2}    \| u \|_{\tilde{L}^2_T B^{s+\alpha,q}_2} \Bigl( \sum_{j} 2^{jq (1/2-\alpha)}\Bigr)^{1/q} \le  C(\alpha)\, 
   \| u \|_{\tilde{L}^2_T B^{0,\infty}_2}  \| u \|_{\tilde{L}^2_T B^{s+\alpha,q}_2} \; ,
  $$
  where $ C(\alpha)>0 $ only depends on $ \alpha>1/2$.\\
 In view of Lemma \ref{popo}, this proves that for $ \alpha> 1/2 $, 
  \begin{equation}
 \label{toto}
  \|\mathcal{L}_{\alpha}(u^2)\|_{E_{\alpha,T}^{s,q}} \lesssim    \| u \|_{\tilde{L}^2_{Tx}}  \| u \|_{\tilde{L}^2_T B^{s+\alpha,q}_2}
  \lesssim \|u\|_{E^{-\alpha,2}_{\alpha,T}}\|u\|_{E^{s,q}_{\alpha,T}} \; ,
  \end{equation}
  where the implicit constants only depends on $ \alpha$.
     In the same way, for any $ \alpha\in ]1/2,1] $, there exists $ C_\alpha>0 $ such that 
   \begin{eqnarray}  
  \|\mathcal{L}_{\alpha}(u v)\|_{E_T^{s,\alpha}} & \lesssim &    \| u \|_{\tilde{L}^2_{Tx}}  \| v \|_{\tilde{L}^2_T B^{s+\alpha,q}_2}
  +  \| v \|_{\tilde{L}^2_{Tx}}  \| u \|_{\tilde{L}^2_T B^{s+\alpha,q}_2} \nonumber \\
  & \le  & C_\alpha \Bigl(  \|u\|_{E^{-\alpha,2}_{\alpha,T}}\|v\|_{E^{s,q}_{\alpha,T}}
 + \|v\|_{E^{-\alpha,2}_{\alpha,T}}\|u\|_{E^{s,q}_{\alpha,T}} \Bigr) \; . \label{toto2}
  \end{eqnarray}
  Let us now  fixed $ \alpha\in ]1/2,1] $.  \eqref{toto2} together with \eqref{2} lead to  the existence of $ \beta>0 $ such that for all
   $ u_0\in B^{-\alpha,2}_2(\R) $ with
 \begin{equation} \label{yy}
  \|u_0\|_{B^{-\alpha,2}_2(\R)}\le \beta\; ,
  \end{equation}
   the mapping
 $$
 u \mapsto S_{\alpha}(\cdot)u_0 + {\mathcal L}_{\alpha}(u^2)
 $$
 is a strict contraction in the  ball of $ E^{-\alpha,2}_{\alpha,1} $  centered at the origin of radius $ (2C_\alpha)^{-1} $. 
   Noticing that $E^{s,q}_{\alpha,1} \hookrightarrow E^{-\alpha,2}_{\alpha,1}$  as soon as 
    \begin{equation}\label{cond}
  ( s\ge -\alpha \mbox{ and  }  1\le q\le 2)\;  \mbox{ or }  \;  s>-\alpha ,
   \end{equation}
 this ensures that the above mapping is also strictly contractive is a small ball of $E^{s,q}_{\alpha,T} $ as soon as
   \eqref{cond}-\eqref{yy}Ê are satisfied.
 Since $S_\alpha $ is a continuous semi-group in $ B^{s,q}_2(\R) $ and  according to Lemma \ref{popo},
 $ E^{s,q}_{\alpha,1}\hookrightarrow \tilde{L}^\infty_1 B^{s,q}_2 $, this leads to the well-posedness result in $ B^{s,q}_2(\R) $  under conditions \eqref{cond}
 for initial data  satisfying \eqref{yy}. The result for general
 initial data follows
  by a simple dilation argument. Indeed, the equation \eqref{NLH} is invariant under
  the dilation $
   u(t,x) \mapsto u_\lambda(t,x)=\lambda^{2\alpha} u (\lambda^{2\alpha} t,\lambda x) $
   whereas $\|\lambda^{2\alpha} u_0(\lambda \cdot) \|_{B^{-\alpha,2}_2(\R)}
    \le \lambda^{\alpha-1/2} \|u_0\|_{B^{-\alpha,2}_2(\R)} \to 0 $ as $ \lambda\searrow 0$.
    Classical arguments then lead to the well-posedness result
    in $ B^{s,q}_2(\R)$ for arbitrary large initial data with a minimal time of existence
    $ T\sim \big(1+\|u_0\|_{B^{-\alpha,2}_2(\R)}\big)^{-\frac{4\alpha}{2\alpha-1}}$.
    Note that,  the well-posedness being obtained by a fixed point argument,
    as a by-product we get that the solution-map : $ u_0 \mapsto u $
     is real analytic from $ B^{s,q}_2(\R) $ into $ C\big([0,T]; B^{s,q}_2(\R)\big).$\\

   \section{Ill-posedness results for $1/2<\alpha\le 1$.}

In this section we prove
  discontinuity results on the flow map $u_0 \mapsto
u(t)$ for any fixed $t>0$ less than some $T>0.$ To clarified the presentation we separate
the case
$ s<-\alpha $ and the case  $ s=-\alpha $ and $ q>2 $.

\subsection{The case $s<-\alpha$}
We take the counter example of \cite{MV} used
for the KdV-Burgers equation.

We define the sequence of initial data $\{\phi_N\}_{N\geq 1}\subset
C^\infty(\R) $ via its Fourier transform by
\begin{equation}\label{deftphij}
\hat{\phi}_N(\xi)=N^{\alpha}\Big(\chi_{I_N}(\xi)+\chi_{I_N}(-\xi)\Big),
\end{equation}
where $I_N=[N, N+2]$ and $\chi_{I_N}$ is the characteristic function
of the interval $I_N,$ $$\chi_{I_N}(\xi)= \begin{cases}
1& \text{ if } \xi \in I_N,\\
0& \text{ if } \xi \not \in I_N.
\end{cases}$$
That is
$$\phi_N(x)= \begin{cases}
{N^{\alpha}\over \pi}{\sin(x)\over  x}\cos\big[(N+1)x\big]& \text{ if } x \not=0,\\
{N^{\alpha}\over \pi}& \text{ if } x=0.
\end{cases}$$
Clearly $\phi_N\in C_0(\R):=\Big\{f\in C(\R)| \lim_{|x|\to
\infty}f(x)=0\Big\}.$

 For any $ (s,q) \in \R\times [1,+\infty] $ we have
$\|\phi_N\|_{B^{s,q}_2(\R)}\sim N^{\alpha+s}  $
 and thus  $\|\phi_N\|_{B^{-\alpha,q}_2(\R)}\sim 1 $ whereas
$\phi_N\to 0$ in $B^{s,q}_2(\R)$,  for $s<-\alpha.$

Let us consider the following bilinear operator, closely related to second iteration of the Picard scheme,
$$A_2(t,h,h)=2 \int_0^tS_{\alpha}(t-t')[S_{\alpha}(t')h]^2dt',$$
where $S_{\alpha}$ is the semi-group of the linear heat equation.
Let us denote by ${\mathcal F}_x$ the partial Fourier transform with
respect to $x.$ Recall that $${\mathcal
F}_x\big(S_{\alpha}(t)\varphi\big)
(\xi)=e^{-t[\xi|^{2\alpha}}{\mathcal F}_x(\varphi)(\xi), \forall\,
\varphi\in {\mathcal S}'(\R),$$ and ${\mathcal F}_x(fg)={\mathcal
F}_x(f)\star{\mathcal F}_x(g),$ where $\star$ is the convolution
product.

It follows that

\begin{eqnarray}
{\mathcal
F}_x\Big(A_2(t,\phi_N,\phi_N)\Big)(\xi)&=&2\int_0^te^{-(t-t')|\xi|^{2\alpha}}\Big(\int_\R
e^{-t'|\xi_1|^{2\alpha}}\hat{\phi}_N(\xi_1)e^{-t'|\xi-\xi_1|^{2\alpha}}\hat{\phi}_N(\xi-\xi_1)d\xi_1\Big)dt'\nonumber\\
&=&2 \int_\R\hat{\phi}_N(\xi_1)\hat{\phi}_N(\xi-\xi_1)\Big(\int_0^t
e^{-(t-t')|\xi|^{2\alpha}}e^{-[|\xi_1|^{2\alpha}+|\xi-\xi_1|^{2\alpha}]t'}dt'\Big)d\xi_1 \nonumber\\
&=& 2\int_\R \hat{\phi}_N(\xi_1)\hat{\phi}_N(\xi-\xi_1)
\Big({e^{-[|\xi_1|^{2\alpha}+ |\xi-\xi_1|^{2\alpha}]t
}-e^{-|\xi|^{2\alpha}t}\over \Theta_\alpha(\xi,\xi_1)}\Bigr)d\xi_1, \label{za}
\end{eqnarray}
where
$$
\Theta_\alpha(\xi,\xi_1)=|\xi|^{2\alpha}-|\xi_1|^{2\alpha}-|\xi-\xi_1|^{2\alpha}\; .
$$
Note that the integrand is nonnegative. In particular, ${\mathcal
F}_x\Big(A_2(t,\phi_N,\phi_N)\Big)(\xi)=|{\mathcal
F}_x\big(A_2(t,\phi_N,\phi_N)\big)(\xi)|.$ Let
$$K_1^N(\xi)=\Big\{\xi_1\;|\; (\xi-\xi_1, \xi_1)\in I_N\times I_N \mbox{ or }   (\xi-\xi_1, \xi_1)\in I_{-N}\times I_{-N}
\Big\}
$$ and
$$K_2^N(\xi)=\Big\{\xi_1\;|\; (\xi-\xi_1, \xi_1)\in I_N\times I_{-N} \mbox{ or }   (\xi-\xi_1, \xi_1)\in I_{-N}\times I_{N}
\Big\}$$
For any $ |\xi |\le \frac{1}{2}$, $ K_1^N(\xi) =\emptyset $ and thus
\begin{eqnarray*}
{\mathcal F}_x\Big(A_2(t,\phi_N,\phi_N)\Big)(\xi)&=&2\int_{K_2^N(\xi)}
\hat{\phi}_N(\xi_1)\hat{\phi}_N(\xi-\xi_1)
\Big({e^{-[|\xi_1|^{2\alpha}+ |\xi-\xi_1|^{2\alpha}]t }-e^{-|\xi|^{2\alpha}t}\over
\Theta_\alpha(\xi,\xi_1)}\Bigr)d\xi_1 \; .
\end{eqnarray*}

On the other hand, for any $(a,b) \in \R_+\times \R_- $ one has obviously,
$$
 |a|^{2\alpha} +|b|^{2\alpha}-|a+b|^{2\alpha}\ge  \big(|a|\wedge
 |b|\big)^{2\alpha}.
 $$
Moreover, it is easy to check that
  $|K_2^N(\xi)|\geq 1 $ and that  in  $K_2^N(\xi)$ it holds $
   N^{2\alpha}\le |\Theta_\alpha(\xi,\xi_1)|\le 2 (N+2)^{2\alpha}$ Hence, fixing $ t\in ]0,1[$, it holds
\begin{eqnarray*}
{\mathcal F}_x\Big(A_2(t,\phi_N,\phi_N)\Big)(\xi)&\ge &
e^{-t/2}
N^{2\alpha}{1-e^{-N^{2\alpha}t}\over 2(N+2)^{2\alpha}}\, \ge \,  {1\over 4}e^{-t/2}, \quad \forall \xi\in [-1/2,1/2] ,
\end{eqnarray*}
 for any $ N >0 $ large enough. This ensures that for any fixed $ (s,q)\in \R \times [1,+\infty] $ and any fixed $t\in ]0,1[ $,
 \begin{equation} \label{estB}
 \| A_2(t,\phi_N,\phi_N)\|_{B^{s,q}_2} \ge  {1\over 4}e^{-t/2}
 \end{equation}
 for $ N>0 $ large enough. Taking $ s<-\alpha $ this proves the discontinuity of the map $ u_0 \mapsto u(t) $ in $ B^{s,q}_2$.  To prove the discontinuity with value in ${\mathcal D}'(\R)$, we proceed as follows.
Let $g\in {\mathcal S}(\R)$ be such that  $ \hat{g}$ is positive equal
to $1$ on $[-1/4,1/4]$ and  supported in $[-1/2,1/2].$ We obtain for $ N >0 $ large enough,

$$|\int_{\R} A_2(t,\phi_N,\phi_N)(x)g(x) dx| \ge \frac{1}{8}e^{-t/4}.$$
On the other hand the analytical well-posedness ensures that $ A_2(t,\phi_N,\phi_N) $ is bounded in $ B^{-\alpha,1}_2 $ uniformly in $ N $.
Then, since ${\mathcal D}(\R)$  is dense in ${\mathcal S}(\R),$
there exists $\varphi\in {\mathcal D}(\R)$ such that
 \begin{equation} \label{estB2}
\Bigl|\int_{\R} A_2(t,\phi_N,\phi_N)(x)\varphi(x) dx\Bigr| \ge
\frac{1}{2^4}e^{-t/4}.
 \end{equation}

This shows that $A_2(t,\phi_N;\phi_N)$ does not converge to $ 0 $ in
${\mathcal D}'(\R)$ and proves the discontinuity from
$B^{s,q}_2(\R), s<-\alpha$ into ${\mathcal D}'(\R).$

We now turn to prove the discontinuity of the flow-map
$$\begin{array}{ccccc}
u(t,\cdot) & : & B^{s,q}_2(\R) & \longrightarrow & B^{s,q}_2(\R) \\
& & h & \longmapsto & u(t,h)= S_{\alpha}(t)h + \int_0^t S_{\alpha}(t-\sigma)
 \big(u^2(\sigma)\big) d\sigma.\\
\end{array}$$
By the theorem of well posedness, there exist $T>0$ and
$\epsilon_0>0$ such that for any $0<\epsilon\leq \epsilon_0,$
$\|h\|_{B^{-\alpha,1}_2}\leq 1$ and $0\leq t\leq T,$
$$u(t,\epsilon h)=\epsilon
S_{\alpha}(t)h+\sum_{k=2}^{\infty}\epsilon^kA_k(t,h^k),$$ where
$h^k=(h,\cdots, h),\; h^k \mapsto A_k(t,h^k)$ are $k-$linear
continuous maps from $(B^{-\alpha,1}_2(\R))^k$ into $C([0,T];
B^{-\alpha,1}_2(\R))$ and the series converges absolutely in
$C([0,T]; B^{-\alpha,1}_2(\R)).$

Hence
$$u(t,\epsilon \phi_N)-\epsilon^2A_2(t,\phi_N,\phi_N)= \epsilon
S_{\alpha}(t)\phi_N+\sum_{k=3}^{\infty}\epsilon^kA_k(t,\phi_N^k).$$ Using the
inequalities
$$\|S_{\alpha}(t)\phi_N\|_{B^{s,1}_2(\R)}\leq \|\phi_N\|_{B^{s,1}_2(\R)}\leq 2N^{s+\alpha}$$
and
$$\Big\|\sum_{k=3}^{\infty}\epsilon^k A_k(t,\phi_N^k)\Big\|_{B^{-\alpha,1}_2(\R)}
\leq \Big({\epsilon\over \epsilon_0}\Big)^3
\Big\|\sum_{k=3}^{\infty}\epsilon_0^kA_k(t,\phi_N^k)\Big\|_{B^{-\alpha,1}_2(\R)}\leq
C\epsilon^3,$$ where $C$ is a positive constant,  we deduce that for $ s\le -\alpha $,
\begin{equation}\label{td}
\sup_{t\in[0,T]}\|u(t,\epsilon
\phi_N)-\epsilon^2A_2(t,\phi_N,\phi_N)\|_{B^{s,1}_2(\R)}\leq
C\epsilon^3+2\epsilon_0N^{s+\alpha}.
\end{equation}

According to  (\ref{estB}) this leads, for $\epsilon\leq {C^{-1}
e^{-t/4}\over 2^5}$, to
$$\|u(t,\epsilon
\phi_N)\|_{B^{s,q}_2(\R)}\geq C_0\epsilon^2/2-2\epsilon_0N^{s+\alpha}.$$ By
letting $N\to \infty$ we obtain the discontinuity result  since $u(t,0)=0$
and $\phi_N\to 0$ in $B^{s,q}_2(\R) $ for $ s<-\alpha$. The discontinuity of the flow-map  from
$B^{s,q}_2(\R)$ into ${\mathcal D}'(\R) $ follows in the same way by combining \eqref{estB2} and \eqref{td}.
\subsection{The case $s=-\alpha$ and $ q>2$}
This case is similar to the precedent except that we have to change  a little the  sequence of initial data. Here we take the same sequence as in the work of Iwabuchi and Ogawa \cite{IO}.
For any $ N \ge 10  $ we define
$$
\psi_N=N^{-{1\over 2}} \sum_{N\le j \le 2N } \phi_{2^j} \; .
$$
where $ \phi_{2^j} $ is defined in \eqref{deftphij}. \\
Noticing that $ \Delta_k \phi_{2^j} = \delta_{k,j} \phi_{2^j} $, we can easily check that
$$
\|\psi_N \|_{B^{-\alpha,q}_2} \sim N^{-{1\over 2} +{1\over q}} \; .
$$
In particular, $ \|\psi_N \|_{B^{\alpha,q}_2}\to 0 $ for any $ q>2 $ whereas $ \|\psi_N \|_{B^{-\alpha,2}_2}=\|\psi_N \|_{H^{-\alpha}}\sim 1 $. Since the equation is analytically well-posed in $ H^{-\alpha}(\R) $, in view of the preceding case, it suffices to prove that $ A_2(\psi_N,\psi_N ,t) $ does not  tend to $ 0 $ in $ {\mathcal D}' $. By the localization, it holds $ \phi_{2^j} \star \phi_{2^{j'}} \equiv 0 $ on $ ]-1/2,1/2[ $ as soon as $ j\neq j' \ge 10$  and the same reasons as above lead to
\begin{eqnarray*}
{\mathcal F}_x\Big(A_2(t,\psi_N,\psi_N)\Big)(\xi)&=& N^{-{1\over 2}}\sum_{N\le j\le 2N} \int_{K_2^{2^j}(\xi)}
\hat{\phi}_{2^j}(\xi_1)\hat{\phi}_{2^j}(\xi-\xi_1)
\Big({e^{-[|\xi_1|^{2\alpha}+ |\xi-\xi_1|^{2\alpha}]t }-e^{-|\xi|^{2\alpha}t}\over
\Theta_\alpha(\xi,\xi_1)}\Bigr)d\xi_1\\
&\ge & N^{-{1\over 2}} e^{-t/2} N^{{1\over 2}}
N^{2\alpha}{1-e^{-N^{2 \alpha }t}\over 2(N+2)^{2\alpha}}\, \ge \,
{1\over 4}e^{-t/2}, \quad \forall \xi\in [-1/2,1/2] ,
\end{eqnarray*}
 for any $ N >0 $ large enough. This completes the proof of the ill-posedness results for $ 1/2<\alpha\le 1$.

\section{Further remarks}\label{periodic} \setcounter{equation}{0}
\subsection{Wellposedness results in the case $ 0<\alpha\le 1/2$}
In this case we only consider the well-posedness results in the Sobolev spaces $ H^s(\R) $.
We prove by standard parabolic methods that one can reach the dilation critical
Sobolev exponant $ s_c=1/2-2\alpha $ except in the case $ \alpha=1/2 $ where
$ 1/2-2\alpha=-\alpha $.
 See for instance \cite{R2}, \cite{G} or \cite{W} for the same kind of results
 in the case $ \alpha=1$.
\begin{To}
\label{wellstandrd} Let  $(\alpha,s) \in \R^2$ be such that   $\alpha\in (0,1/2]$ and
$s\ge 1/2-2\alpha $ with $ s>-\alpha $.
Then the Cauchy problem \eqref{NLH} is locally well-posed in $ H^s(\R) $.
\end{To}

\begin{proof} The proof is done using a fixed point argument on a suitable metric space.
The case $ s>1/2 $ is trivial since $ H^s(\R) $ is an algebra and
the semi-group $ S_\alpha $ is contractive on $ H^s(\R)$. One can
thus simply perform a fixed point argument in $ C([0,T]; H^s(\R))$
on the Duhamel formula  for a suitable $ T>0 $ related to $
\|u_0\|_{H^s(\R)} $.  The case $ s=1/2 $ is also rather easy and is
postponed  at the end of the proof. So let us assume that
\begin{equation}\label{ra}
1/2-2\alpha \le s<1/2 \mbox{ if }  0<\alpha<1/2 \mbox{ and  } -1/2<s<1/2 \mbox{ if }  \alpha=1/2 \; ,
\end{equation}
that is $\alpha,\; s$ belong to the set
$$
 \Bigl\{ (\alpha,s)\in\R^2 \, |\, 0<\alpha\le 1/2, \; s\ge 1/2-2\alpha \;
 \mbox{and } s>-\alpha \; \Bigr\}.
$$
For $ s$ fixed as above we take $ 0<s_0<1/2 $ such that
$$
0<s_0-s<\alpha\quad \mbox{and } \quad 2s_0-{1\over 2}<s
  \;  .
 $$
  This is obviously possible  for $ \alpha=1/2 $ since $ s>-1/2 $, and  for $ 0<\alpha<1/2 $ since
$ s+1/2>s+\alpha\ge (1/2-2\alpha)+\alpha=1/2-\alpha>0 $. We first establish the existence and uniqueness of a solution of  (\ref{NLHint})  in
$$X_{M,T}:=\Big\{u\in C\big((0,T], H^{s_0}(\R)\big)\; |\; \|u\|_{X_{T}}:=\sup_{t\in(0,T]}t^{{s_0-s\over
2\alpha}}\|u(t)\|_{H^{s_0}(\R)}\leq M\Big\} $$
by proving that the mapping
$$\Lambda_{u_0}(u)(t)= S_{\alpha}(t)u_0 \mp \int_0^t S_{\alpha}(t-\sigma)
\big(u^2(\sigma)\big) d\sigma,$$ is a strict contraction in $
X_{M,T} $ for suitable $M>0,\; T>0$.

From classical regularizing  effects for the fractional heat equation it holds
\begin{equation} \label{effetregularisantHs}
\|S_\alpha(t)f\|_{H^{s_2}(\R)}\leq C t^{-{s_2-s_1\over
2\alpha}}\|f\|_{H^{s_1}(\R)}, \forall \; s_1\leq s_2,\; \forall\;
f\in H^{s_1}(\R).\end{equation} Applying \eqref{effetregularisantHs}
with $ (s_1,s_2)=(s,s_0) $, yields
\begin{equation}\label{linear2}
t^{s_0-s\over 2 \alpha}  \| S_{\alpha}(t) u_0\|_{H^{s_0}(\R)}
\lesssim \|u_0 \|_{H^s(\R)} \; .
\end{equation}
Now, according to \cite{R1}, since $0<s_0<1/2$,  it holds :
\begin{equation} \label{ineprod} \|u v \|_{H^{2s_0-{1\over 2}}(\R)}\leq
C\|u\|_{H^{s_0}(\R)} \|v\|_{H^{s_0}(\R)}, \end{equation} where $C$ is a positive constant.
We thus obtain for any $ t>0 $,
\begin{eqnarray}
t^{s_0-s\over 2 \alpha}  \Bigl\| \int_0^t S_{\alpha}(t-t') u^2(t')
\, dt' \Bigr\|_{H^{s_0}(\R)} & \lesssim &
 t^{s_0-s\over 2 \alpha} \int_0^t  \Bigl\| S_{\alpha}(t-t') u^2(t')  \Bigr\|_{H^{s_0}(\R)} \, dt' \nonumber \\
 & \lesssim &  t^{s_0-s\over 2 \alpha} \int_0^t  (t-t')^{s_0-1/2\over 2\alpha}
 \|u^2(t')\|_{H^{2s_0-1/2}(\R)} \, dt' \nonumber \\
& \lesssim &   t^{s_0-s\over 2 \alpha} \int_0^t  (t-t')^{s_0-1/2\over 2\alpha} \|u(t')\|_{H^{s_0}(\R)}^2 \, dt' \nonumber \\
& \lesssim & \sup_{\tau\in ]0,t[} \Bigl(\tau^{{s_0-s\over
2\alpha}}\|u(\tau)\|_{H^{s_0}(\R)} \Bigr)^2  \; t^{s-(1/2-2\alpha)\over 2\alpha} \int_0^1 (1-\theta)^{2s_0-1\over 4 \alpha} \theta^{s-s_0\over \alpha} \, d\theta \nonumber \\
 & \lesssim &     t^{s-(1/2-2\alpha)\over 2\alpha} \|u\|_{X_t}^2 \; \label{duha}
\end{eqnarray}
 where in the last step we used that $ 0<s_0-s<\alpha $ and that $ s_0>1/2-2\alpha $ since $ s_0>s\ge 1/2-2\alpha $.
 In view of \eqref{duha} we easily get  for $ 0<T<1 $ and $v_i\in X_T $, $ i=1,2$,
\begin{equation}\label{l1}
 \| \Lambda_{u_0}(v_i) \|_{X_T} \lesssim  \|S_{\alpha}(\cdot) u_0\|_{X_{T}} + T^{s-(1/2-2\alpha)\over 2\alpha} \|v_i\|_{X_T}^2
 \end{equation}
 and
 \begin{equation}\label{l2}
 \| \Lambda_{u_0}(v_1-v_2) \|_{X_T} \lesssim  T^{s-(1/2-2\alpha)\over 2\alpha} (\|v_1\|_{X_T} +\| v_2\|_{X_T})\|v_1-v_2\|_{X_T} \; .
 \end{equation}
 Combining these estimates with \eqref{linear2} we infer  that for $ s>1/2-2\alpha $,
 $ \Lambda_{u_0} $ is a strict contraction on $ X_{M,T} $ with $ M\sim  \|u_0\|_{H^s(\R)} $
 and $ T\sim \|u_0\|_{H^s(\R)}^{-2\alpha\over s-(1/2-2\alpha)} $ if $ s>1/2-2\alpha $.
 This leads to the  existence and uniqueness in $ X_T $  for any $ u_0\in H^s(\R) $.
 For $ s=1/2-2\alpha , $ $ \Lambda_{u_0} $ is also  a strict contraction on $ X_{M,T} $
 with $ M\sim  \|u_0\|_{H^s(\R)} $ and $ T\sim 1 $ but only  under a smallness assumption on
 $ \|u_0\|_{H^s(\R)} $. Hence, we  get the existence in  $ X_T $  for any $ u_0\in H^s(\R) $
 with  small initial data.
 Now to prove that the solution $ u$ belongs to $  C([0,T] ; H^s(\R)) $
 we first notice that $ S_\alpha(u_0) \in  C(\R_+; H^s(\R)) $. Moreover , according to
  \eqref{effetregularisantHs}, we have
 \begin{align}
 \sup_{t\in ]0,T[} \Bigl\| \int_0^t S_{\alpha}(t-t') (u^2-v^2)(t') \, dt' \Bigr\|_{H^{s}(\R)} &
  \lesssim
  \sup_{t\in ]0,T[} \int_0^t  \Bigl\| S_{\alpha}(t-t') (u^2-v^2)(t')  \Bigr\|_{H^{s}(\R)} \, dt'
   \nonumber \\
 & \hspace*{-10mm}\lesssim   \sup_{t\in ]0,T[} \int_0^t
 (t-t')^{\min(0,{(2s_0-1/2)-s\over 2\alpha})} \|u^2-v^2\|_{H^{2s_0-1/2}(\R)} \, dt' \nonumber
 \\
& \hspace*{-10mm}\lesssim    \sup_{t\in ]0,T[} \int_0^t
(t-t')^{\min(0,{(2s_0-1/2)-s\over 2\alpha})} \|u-v\|_{H^{s_0}(\R)}\|u+v\|_{H^{s_0}(\R)} \,
dt' \nonumber \\
& \hspace*{-10mm}\lesssim  \|u+v\|_{X_T}\|u-v\|_{X_T} \nonumber \\
 & \hspace*{-5mm} \; T^{\min(1+{s-s_0\over \alpha},{s-(1/2-2\alpha)\over 2\alpha})} \int_0^1
 (1-\theta)^{\min(0,{(2s_0-1/2)-s\over 2\alpha})} \theta^{s-s_0\over \alpha} \, d\theta \nonumber \\
 & \hspace*{-10mm}\lesssim    T^{\min(1+{s-s_0\over \alpha},{s-(1/2-2\alpha)\over 2\alpha})} \|u+v\|_{X_T}\|u-v\|_{X_T} \; \label{duha2}
\end{align}
  where in the last step we used that $ 0<s_0-s<\alpha $ and that
  $ {2 s_0-1/2-s\over 2\alpha}>-1 $ since $ 2 s_0-s>s\ge 1/2-2\alpha $.
   This ensures that starting with a continuous function
   $v\in C\big([0,T]; H^s(\R)\big) \cap X_{M,T} $, the sequence of function
   constructed by the Picard sheme  that converges to the solution in $u\in  X_T $
   is a Cauchy sequence in $C\big([0,T]; H^s(\R)\big) $ and thus
   $ u\in C\big([0,T]; H^s(\R)\big) $.
   The continuous dependence with respect to initial data in
   $ H^s(\R) $ follows  also easily from  \eqref{duha2}.

  It remains to handle the case of arbitrary large initial data in $ H^{s_c}(\R) $
  when $ s=s_c=1/2-2\alpha $. We first notice that, according
  to \eqref{l1}-\eqref{l2}, $\Lambda_{u_0} $ is a strict
  contraction in $ X_{M,T} $ as soon as $ M=2\|S_\alpha(\cdot) u_0\|_{X_T} $ is small enough.
  Then, fixing $ u_0\in H^{s_c}(\R) $, by the density of $ H^{s_0}(\R) $ in $ H^{s_c} (\R) $
  we infer that for any $\varepsilon>0 $ there exists $ u_{0,\varepsilon}\in H^{s_0}(\R) $
  such that
   $ \|u_0-u_{0,\varepsilon}\|_{H^{s_c}(\R)}<\varepsilon$.
   Since $ u_{0,\varepsilon}\in H^{s_0}(\R) $ it holds
   $\|S_\alpha(\cdot) u_{0,\varepsilon}\|_{X_T} \le
   T^{s_0-s\over 2\alpha}  \|u_{0,\varepsilon}\|_{H^{s_0}(\R)} $. This leads to
   $$
   \|S_\alpha(\cdot) u_0\|_{X_T} \lesssim   T^{s_0-s\over 2\alpha}
    \|u_{0,\varepsilon}\|_{H^{s_0}(\R)}+\varepsilon \; .
   $$
   Noticing that the right-hand side member of the above inequality can be made arbitrary
   small by choosing suitable $ \varepsilon>0 $ and $ T>0 $,
   this proves the local  existence in  $C\big([0,T];H^{s_c}(\R)\big)\cap X_T $
   for arbitrary large initial data in $H^{s_c}(\R) $.
   Note that here $ T>0 $ does not depend only on
    $\|u_0\|_{H^{s_c}(\R)} $ but on the Fourier profile of $ u_0$.
    The uniqueness holds in
     $ \{f\in X_T\, /\,  \|f\|_{X_t} \to 0\;  \mbox{as }\;  t\searrow 0 \} $  .
     This completes the proof for $ (\alpha, s) $ satisfying \eqref{ra}.

  Finally for $s=1/2$ we  apply the fixed point argument in
  $$\tilde{X}_{M,T}:=\Big\{u\in C\big((0,T], H^{1+\alpha\over 2 }(\R)\big)\; |\; \|u\|_{\tilde{X}_{T}}:=\sup_{t\in(0,T]}t^{{1\over
4}}\|u(t)\|_{H^{1+\alpha\over 2}(\R)}\leq M\Big\}.$$ Using that $
H^{1+\alpha\over 2}(\R) $ is an algebra we easily get
    \begin{eqnarray}
t^{1\over 4}  \Bigl\| \int_0^t S_{\alpha}(t-t') u^2(t') \, dt'
\Bigr\|_{H^{1+\alpha\over 2}(\R)}
& \lesssim & t^{1\over 4} \int_0^t \|u^2\|_{H^{1+\alpha\over 2}(\R)} \, dt' \nonumber \\
& \lesssim & t^{1\over 4} \int_0^t \|u\|_{H^{1+\alpha\over 2}(\R)}^2 \, dt' \nonumber \\
 & \lesssim &     t^{3\over 4} \|u\|_{\tilde{X}_t}^2  \int_0^1 \theta^{-1/2} \, d\theta\;  .\label{duha4}
\end{eqnarray}
This gives the local existence and uniqueness in $ \tilde{X}_{M,T} $
for $ M\sim \|u_0\|_{H^{1/2}(\R)}$ and $ T\sim
\|u_0\|_{H^{1/2}(\R)}^{-4\over 3} $. The fact that the solution $ u$
belongs to $C\big([0,T]; H^{1\over 2}(\R)\big) $ and the continuous
dependence with respect to initial data in $ H^{1\over 2}(\R)$
follows by noticing that
 \begin{eqnarray}
 \sup_{t\in ]0,T[} \Bigl\| \int_0^t S_{\alpha}(t-t') (u^2-v^2)(t') \, dt'
 \Bigr\|_{H^{1\over 2}(\R)}
 & \lesssim &  \sup_{t\in ]0,T[} \int_0^t
 \|S_{\alpha}(t-t') (u^2-v^2)(t')\|_{H^{1\over 2}(\R)} \, dt' \nonumber \\
 & \lesssim &  \sup_{t\in ]0,T[} \int_0^t
 \|(u^2-v^2)(t')\|_{H^{1\over 2}(\R)} \, dt' \nonumber \\
 & \lesssim &  \sup_{t\in ]0,T[} \int_0^t
 \|u(t')^2-v(t')^2\|_{H^{1+\alpha\over 2}(\R)} \, dt' \nonumber\\
 & \lesssim &  \sup_{t\in ]0,T[} \int_0^t
 \|u(t')+v(t')\|_{H^{1+\alpha\over 2}(\R)}\|u(t')-v(t')\|_{H^{1+\alpha\over 2}(\R)} \, dt' \nonumber\\
& \lesssim & \|u+v\|_{\tilde{X}_T}\|u-v\|_{\tilde{X}_T} \; T^{1\over
2 } \int_0^1 \theta^{-1/2} d\theta \; .
\end{eqnarray}
\end{proof}
\subsection{Illposedness result for $\alpha=1/2 $ and $ s=-1/2$}
Let us now prove an  ill-posedness result at the crossing point $(\alpha,s)=(1/2,-1/2) $ of the two lines $ s=-\alpha $ and $s=1/2-2\alpha $.
 Recall that there exists $ T_0>0 $ and $ R_0>0 $  such that  the solution-map $ u_0 \mapsto u $ associated with \eqref{NLH} for $ \alpha=1/2 $ is well-defined and  continuous from the ball $ B(0,R_0)_{L^2} $ of $ L^2(\R) $ with values in $ C([0,T]; L^2(\R))$.
The following norm inflation result clearly disproves the continuity of this solution map    from  $ B(0,R_0)_{L^2} $ endowed with the $ H^{-1/2} $-topology with values in   $ C([0,T]; H^{-1/2}) $, for any $ T\le T_0 $.
\begin{To}\label{to3}
There exists a sequence $ T_N \searrow 0 $ and a sequence of initial data $ \{\phi_N\} \subset L^2(\R) $ such that  the sequence of emanating solutions
 $ \{ u_N\}  $ of \eqref{NLH}is included in  $ C([0,T_N]; L^2(\R)) $ and  satisfy
 \begin{equation}
 \|\phi_N \|_{H^{-1/2}} \to 0  \quad \mbox{ and } \quad \|u_N(T_N) \|_{H^{-1/2}} \to +\infty \mbox{ as } N\to \infty \; .
 \end{equation}
\end{To}

 We follow exactly the very nice proof of Iwabuchi-Ogawa \cite{IO} that proved the ill-posedness in $ H^{-1} $ of the
 2-D quadratic heat equation. Note that $ (1,-1) $ is  the intersection of the two lines $ s=-\alpha $ and
 $ s=1-2\alpha $, this last line corresponding to the scaling critical Sobolev exponent in dimension 2.
We need to introduce the rescaled modulation spaces $ (M_{2,1})_N $ that are defined for any  integer $ N\ge 1 $ by
$$
(M_{2,1})_N:=\Bigl\{ u\in {\mathcal S}'(\R) \, |\,
\|u\|_{(M_{2,1})_N} <\infty\Bigl\}
$$
where
$$
 \|u\|_{(M_{2,1})_N} :=\sum_{k\in 2^N \Z} \| \hat{u} \|_{L^2(k,k+2^N)} \; .
 $$
 It is easy to check that
 \begin{eqnarray}
 \|u v  \|_{ (M_{2,1})_N}  & = &\sum_{k\in 2^N \Z} \| \hat{u}\ast \hat{v} \|_{L^2(k,k+2^N)}  \nonumber \\
 & \lesssim & \Bigl( \sum_{k\in 2^N \Z} \| \hat{v} \|_{L^1(k,k+2^N)} \Bigr) \Bigl( \sum_{k\in 2^N \Z} \| \hat{u} \|_{L^2(k,k+2^N)} \Bigr) \nonumber \\
 & \le & C_0\, 2^{N/2} \|u   \|_{ (M_{2,1})_N} \| v  \|_{ (M_{2,1})_N} \label{mult} \; ,
 \end{eqnarray}
 for some constant $ C_0>0$.
 Hence $ (M_{2,1})_N $ is an algebra and, since $ S_\alpha $ is clearly continuous in $(M_{2,1})_N$, we easily get for any $ u_0\in (M_{2,1})_N$
  and any $ v\in L^\infty_T (M_{2,1})_N $ that
 \begin{equation}
\| \Lambda_{u_0}(v)  \|_{L^\infty_T (M_{2,1})_N} \lesssim \|u_0\|_{(M_{2,1})_N}+ T 2^{N/2} \| \Lambda_{u_0}(v)  \|_{L^\infty_T (M_{2,1})_N}^2 \; .
 \end{equation}
Picard iterative scheme then ensures the well-posedness of \eqref{NLH} in  $ (M_{2,1})_N $ with a minimal time of existence
 \begin{equation}\label{minimaltime}
 T\sim 2^{-N/2}\| u_0 \|^{-1}_{(M_{2,1})_N } \; .
 \end{equation}
Therefore the analytic  expansion \eqref{analytic}  holds in $
(M_{2,1})_N $ on the time interval $ [0,T]$.

We set
$$
\widehat{\phi_{N,R}}:=R \varphi (2^{-N} \cdot)
$$
where $ \varphi $ is  defined in the beginning of Section 2, $ N\ge 1 $ and $ R>0 $ tends to $ 0$ as $ N\to \infty $.
We easily check that
\begin{equation}\label{phi}
 \|\phi_{N,R} \|_{ (M_{2,1})_N} \le  4 R 2^{N/2} \quad \mbox{and  } \quad\|\phi_{N,R} \|_{H^{-1/2}} \sim R \to 0\; \mbox{as } N\to +\infty\; .
\end{equation}
According to \eqref{minimaltime},
 the solution $ u_{N,R} $ of \eqref{NLH} emanating from $ \phi_{N,R}$ exists and satisfies on $ [0, 2^{-N}] $,
  \begin{equation}\label{analytic}
 u_{N,R}(t)=
S_{\alpha}(t)\phi_{N,R}+\sum_{k=2}^{\infty}A_k(t,\phi_{N,R}^k),
\end{equation}
where
 $h^k=(h,\cdots, h),\; h^k \mapsto A_k(t,h^k)$ are $k-$linear
continuous maps from $((M_{2,1})_N)^k$ into $C([0,T]; (M_{2,1})_N)$ and
the series converges absolutely in  $C([0,T]; (M_{2,1})_N)$.
Moreover, setting $ A_1(t,h):=S_{\alpha}(t)h $,  the $ A_k $'s satisfy the following recurrence formula for $ k\ge 2$,
 \begin{equation}\label{recurrence}
A_k(t,h^k)= \sum_{k_1+k_2=k} \int_0^t S_\alpha (t-t')\Bigl( A_{k_1}(t',h^{k_1})A_{k_2}(t',h^{k_2})\Bigr)\, dt' \; .
 \end{equation}
According to \eqref{phi}, for any $ t>0 $,
$$
\|S_\alpha(t) \phi_{N,R} \|_{H^{-1/2}} \lesssim R \to 0 \; \mbox{as
} N\to +\infty\; .
$$
Moreover, as in \eqref{za} , we have
$$
\widehat{A_2(t)}(\xi):={\mathcal
F}_x\Big(A_2(t,\phi_{N,R},\phi_{N,R})\Big)(\xi)
= 2\int_\R \widehat{\phi_{N,R}}(\xi_1)\widehat{\phi_{N,R}}(\xi-\xi_1)e^{-|\xi| t}
\Big({e^{[|\xi|-|\xi_1|-|\xi-\xi_1|]t
}-1\over |\xi|-|\xi_1|-|\xi-\xi_1|}\Bigr)d\xi_1 \; .
$$
By the support  property of $ \hat{\phi}_{N,R} $ we infer that for $
t\lesssim 2^{-N} $ it holds
$$
e^{-|\xi| t} \, \Big|{e^{[|\xi|-|\xi_1|-|\xi-\xi_1|]t
}-1\over |\xi|-|\xi_1|-|\xi-\xi_1|}\Bigr| \sim t \; .
$$
This ensures that $ |\widehat{A_2(t)}(\xi)|\gtrsim R^2 2^{N}  t $ \hspace*{2mm} for  $ t\lesssim 2^{-N} $ and $ |\xi|\le2^N/8  $. Hence,
\begin{eqnarray}
\|A_2(t) \|_{H^{-1/2}} & \gtrsim &  R^2 2^{N}  t \Bigl( \int_{-2^N\over 8}^{2^N\over 8}  \langle \xi \rangle^{-1}\Bigr)^{1/2} \nonumber \\
 &\gtrsim &  R^22^N t  N^{1/2}  \; ,\label{A2}
\end{eqnarray}
where $\langle \xi \rangle=(1+|\xi|^2)^{1/2}.$ On the other hand,
we have the following upper bound on the $ H^{-1/2}$-norm of the
$A_k $'s.
\begin{lem} \label{ll}
For any $ k\ge 3 $ it holds
\begin{equation}\label{Ak}
\|A_k(t,\phi_{N,R}^k)\|_{H^{-1/2}}\le  8^k C_0^{k-1} (N+\ln k)^{1/2} R^k 2^{(2k-2){N\over 2}} k  t^{k-1}  \; .
\end{equation}
\end{lem}
\begin{proof}
We first prove that for $ k\ge 1 $ we have
\begin{equation}\label{estM}
\|A_k(t,\phi_{N,R}^k)\|_{(M_{2,1})_N} \le 4^k C_0^{k-1} t^{k-1} R^k
2^{(2k-1)N/2}.
\end{equation}
For $ k=1 $,  it follows directly from \eqref{phi} that
$$
\|A_1(t,\phi_{N,R}) \|_{(M_{2,1})_N}=\|S_{1/2}(t) \phi_{N,R} \|_{(M_{2,1})_N} \le 4 R 2^{N/2} ,
$$
and using \eqref{mult} we obtain
\begin{eqnarray}
\|A_2(t,\phi_{N,R}^2)\|_{(M_{2,1})_N}&  \le &  \int_0^t \Big\| \Bigl( A_1(\tau, \phi_{N,R})\Bigr)^2 \Bigr\|_{(M_{2,1})_N} \, d\tau\nonumber \\
& \le & C_0 2^{N/2}  \int_0^t \Big\| A_1(t) \phi_{N,R} \Bigr\|_{(M_{2,1})_N}^2  \, d\tau \nonumber \\
 & \le & 4^2 C_0 2^{3 N/2} R^2 t  \; .
\end{eqnarray}
In view of the expression \eqref{recurrence} of $ A_k(t,\phi_{N,R}^k)$,  \eqref{estM}  follows then easily  by a recurrence argument on $ k $. \\
Now, again from \eqref{recurrence} it is easy to check that the
support of  the space Fourier transform of $A_k(t,\phi_{N,R}^k) $ is
contained in $ \{\xi\in \R, \, |\xi| \le k 2^{N+2} \}$. It thus
holds, using Hausdorff-Young and H\"older inequalities, that
\begin{eqnarray}
\|A_k(t,\phi_{N,R}^k) \|_{H^{-1/2}} & \le  &  \|\langle \cdot
\rangle^{-1/2}\|_{L^2(-k 2^{N+2},\, k2^{N+2})} \,
\sup_{\xi\in \R} |\widehat{A_k}(t,,\phi_{N,R}^k) |(\xi) \nonumber \\
& \lesssim & 2 (N+\ln k)^{1/2}\sum_{k_1+k_2=k} \int_0^t \|
\widehat{A_{k_1}}(\tau,\phi_{N,R}^{k_1})\star
\widehat{A_{k_2}}(\tau,\phi_{N,R}^{k_2})\|_{L^\infty_\xi} \, d\tau  \nonumber \\
& \le & 2 (N+\ln k)^{1/2}\sum_{k_1+k_2=k} \int_0^t
\|A_{k_1}(\tau,\phi_{N,R}^{k_1})\|_{L^2}
\|A_{k_2}(\tau,\phi_{N,R}^{k_2})\|_{L^2} \, d\tau \; .\nonumber
\end{eqnarray}
Therefore \eqref{mult} and the fact that $ (M_{2,1})_N \hookrightarrow L^2 $, with an embedding constant less than 1, lead to
\begin{eqnarray}
\|A_k(t,\phi_{N,R}^k) \|_{H^{-1/2}} & \le  &  2 (N+\ln
k)^{1/2}\sum_{k_1+k_2=k} \int_0^t
\|A_{k_1}(\tau,\phi_{N,R}^{k_1})\|_{(M_{2,1})_N}
\|A_{k_2}(\tau,\phi_{N,R}^{k_2})\|_{(M_{2,1})_N} \, d\tau \nonumber \\
& \le & 2 (N+\ln k)^{1/2} 4^k C_0^{k-1} R^k 2^{(2k -2)N/2} \int_0^t  \tau^{k-2} \, d\tau  \Big(\sum_{k_1+k_2=k}1\Big) \nonumber \\
& \le &  2 (N+\ln k)^{1/2}  4^k C_0^{k-1}R^k2^{(2k -2)N/2}  \frac{k
\, t^{k-1}}{k-1}  \; .
\end{eqnarray}
\end{proof}
We  deduce from the above lemma that
\begin{eqnarray}
\sum_{k\ge 3} \|A_k(t, \phi_{N,R}^k)\|_{H^{-1/2}} \le  8^3 C_0^3 2^{2N} R^3 t^2  \sum_{k\ge 3} (N+\ln k)^{1/2} (8 C_0 2^{N} Rt)^{k-3} \; .
\end{eqnarray}
Therefore setting $ R:=N^{-1/4} \ln N $ we get
\begin{equation}\label{tu}
\sup_{0<t\le (8C_0 2^{N})^{-1}} \sum_{k\ge 3} \|A_k(t,
\phi_{N,R}^k)\|_{H^{-1/2}} \lesssim  N^{-3/4} (\ln N)^3 \sum_{k\ge
3} (N+\ln k)^{1/2} \Bigl( \frac{\ln N}{N^{1/4}}\Bigr)^{k-3} \le
\gamma(N)\; ,
\end{equation}
with $ \gamma(N)\to 0 $ as $ N\to \infty $. Setting $ T_N:=(8C_0
2^{N})^{-1}$ and gathering \eqref{A2}, \eqref{Ak}, \eqref{tu}   and
\eqref{analytic} we deduce that
\begin{equation}\label{to}
\|u_N(T_N)\|_{H^{-1/2}} \gtrsim C (\ln N)^2 - N^{-1/4} \ln N - \gamma(N)\longrightarrow +\infty \mbox{ as } N\to \infty  \; ,
\end{equation}
which, together with \eqref{phi},  concludes the proof of Theorem \ref{to3}.
\begin{rem}
{\rm By the previous theorems, for  $0<\alpha\le 1 $, we obtained
the well-posedness of the fractional heat equation (\ref{NLHint}) in
$H^s(\R)$ for $s\ge \max(-\alpha, 1/2-2\alpha) $ and
$(\alpha,s)\not=(1/2,-1/2)$. See Figure \ref{resumefig}.}
\end{rem}

\begin{center}
\begin{figure}
\includegraphics[scale=0.7]{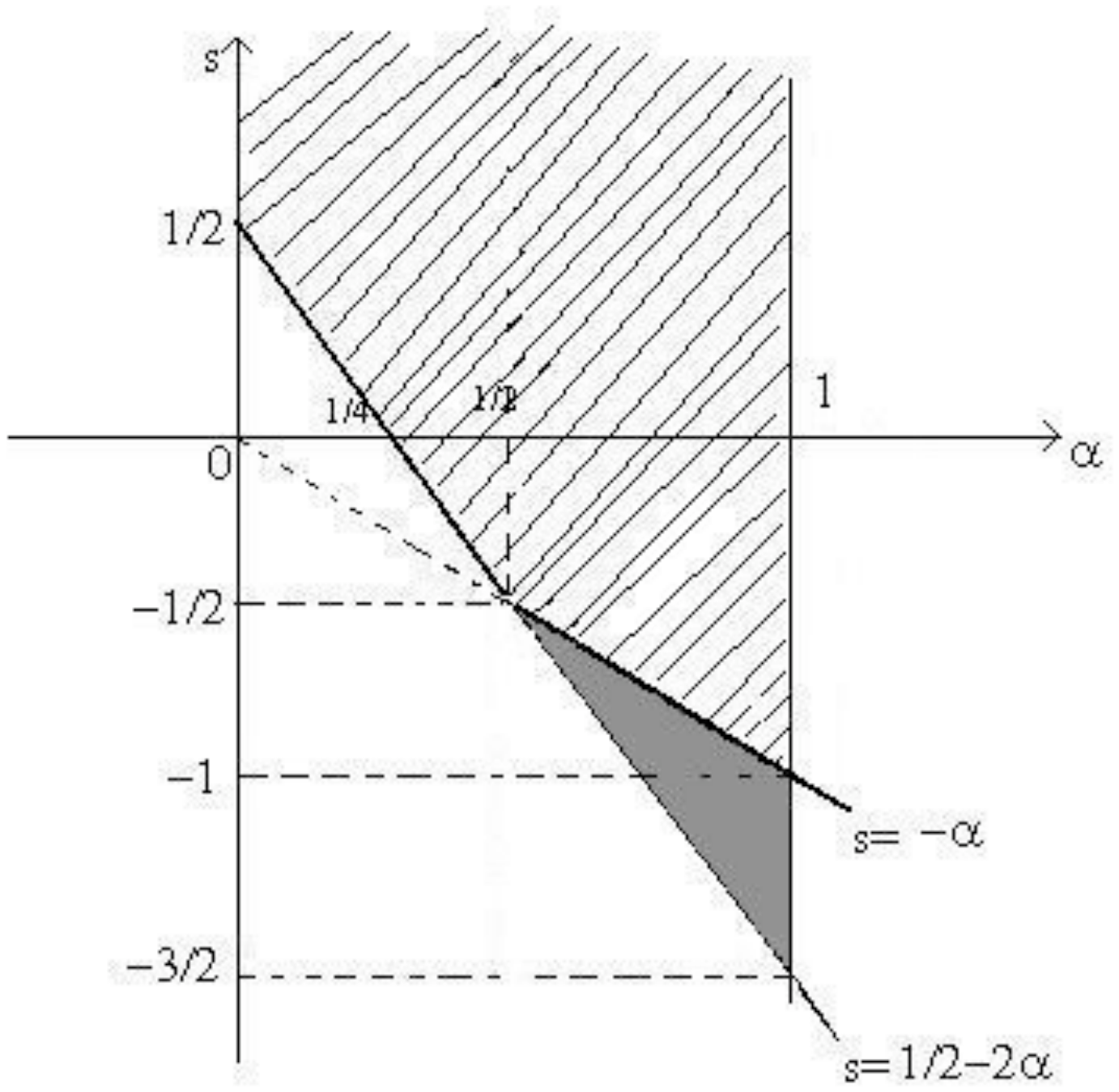}
\caption{The domains of well-posedness and ill-posedness for the
fractional heat equation (\ref{NLHint}) in $H^s(\R),\; \alpha \, \in
(0,1],\; s\in \R.$ Well-posedness holds inside the hatched region
with its boundary without the point $(1/2,-1/2)$. Ill-posedness
holds in side the shaded region and the point $(1/2,-1/2)$.}
\label{resumefig}
\end{figure}
\end{center}

\subsection{The periodic case}
The periodic case can be treated in exactly the same way as the real line case since
the linear fractional   heat equation enjoys the same regularizing effects  on the torus.
The only difference is that the dilation symmetry, that we used at the end of Section
\ref{Well}, does not keep a torus invariant but maps it to another torus.
 To overcome this difficulty it suffices to notice that, in the periodic setting,
 the  estimates derived in Section \ref{linear} are uniform for all  period
 $ \lambda \ge 1 $.


\thebibliography{ww}

\bibitem{BT}{I. Bejenaru and T. Tao, {\it Sharp well-posedness and ill-posedness results
for a quadratic non-linear Schr\"odinger equation}, J. Funct. Anal.
233 (2006), 228--259.}

\bibitem{B}{J. Bourgain, {\it Fourier restriction phenomena for certain lattice subset
applications to nonlinear evolution equation}, Geometric and
Functional Anal., 3(1993), 107--156, 209--262.}

\bibitem{BC}{H. Br\'ezis and T. Cazenave, {\it A nonlinear heat equation with
singular initial data}, J. Anal. Math., 68 (1996), 277-304.}

\bibitem{G}{T. Ghoul, {\it An extension of Dickstein's ``small lambda'' theorem for finite time blowup},
Nonlinear Analysis, 74 (2011), 6105--6115.}

\bibitem{Gi}{Y. Giga, {\it Solutions for semilinear parabolic equations in
$L^p$ and regularity of weak solutions of the Navier-Stokes system},
 J. Differential Equations, 62 (1986), 186–212.}

 \bibitem{HW}{A. Haraux and F. B. Weissler, {\it
Non-uniqueness for a semilinear initial value problem}, Indiana
Univ. Math. J., 31 (1982), 167-189.}

\bibitem{H}{D. Henry, {\it Geometric Theory of Semilinear Parabolic Equations},
Lecture Notes in Mathematics, 840, Springer-Verlag, Berlin-New York,
1981.}

\bibitem{IO} T. Iwabuchi and T. Ogawa, {\it  Ill-posedness for
nonlinear Schr\"odinger equation with quadratic non-linearity in lo dimensions},
 preprint.

\bibitem{MRY}{L. Molinet, F. Ribaud and Y. Youssfi, {\it
Ill-posedness issues for a class of parabolic equations}, Proc. Roy.
Soc. Edinburgh Sect. A, 132 (2002), 1407--1416.}

\bibitem{MV}{L. Molinet and S. Vento, {\it Sharp ill-posedness and
well-posedness results for the KdV-Burgers equation: the real line
case}, Ann. Scuola Norm. Sup. Pisa Cl. Sci.  10(2011), 531--560.}

\bibitem{R1}{F. Ribaud, {\it  Cauchy problem for semilinear parabolic equations with initial data
in $H^s_p({R}^n)$ spaces}, Rev. Mat. Iberoamericana, 14 (1998),
1--46.}

\bibitem{R2}{F. Ribaud, {\it Semilinear parabolic equations with distributions as initial
data}, Discrete Contin. Dynam. Systems,  3 (1997), 305--316.}

\bibitem{STW1}{S. Snoussi, S. Tayachi, and F.B. Weissler,  {\it Asymptotically self-similar
global solutions of a semilinear parabolic equation with a nonlinear
gradient term,} Proc. R. Soc. Edinb., Sect. A, Math. 129 (1999),
1291--1307.}

\bibitem{STW2}{S. Snoussi, S. Tayachi, and F.B. Weissler,  {\it Asymptotically self-similar global solutions of a general semilinear
heat equation}, Math. Ann., 321 (2001), 131--155.}

\bibitem{TX}{Z. Tan and Y. Xu, {\it Existence and nonexistence of global solutions for a semilinear heat equation
with fractional laplacien}, Acta Mathematica Scientia, Ser. B,
32(2012), 2203--2210.}


\bibitem{tataru} {D. Tataru, {\sl On global existence and scattering for the wave maps equation},
Amer. J. Math. 123 (2001),  37--77.}

\bibitem{Taylor} {M. E. Taylor, {\sl Pseudodifferential Operators and Nonlinear PDE},
Progress in Mathematics, 100. Birkha\"user Boston, Inc., Boston, MA,
1991.}

\bibitem{W1}{F. B. Weissler, {\it Semilinear evolution equations in Banach spaces}, J. Funct. Anal. 32 (1979),
277--296.}


\bibitem{W2}{F. B. Weissler, {\it Local existence and nonexistence for semilinear parabolic equation in
$L^p$}, Indiana Univ. Math. J., 29 (1980), 79--102.}

\bibitem{W3}{F. B. Weissler, {\it Existence and nonexistence of global solutions for
a semilinear heat equation},  Israel J. Math., 38  (1981), 29--40.}

\bibitem{W}{J. Wu, {\it Well-posedness of a semilinear heat equation with
weak initial data}, J. Fourier Anal. Appl., 4(1998), 629--642.}



\endthebibliography
\end{document}